\documentclass{amsart}
\usepackage{color}
\usepackage{graphicx,epsf}
\usepackage{amsmath,amscd}
\newtheorem{theorem}{Theorem}[section]
\newtheorem{lemma}[theorem]{Lemma}
\newtheorem{corollary}[theorem]{Corollary}
\newtheorem{proposition}[theorem]{Proposition}
\theoremstyle{definition}
\newtheorem{definition}[theorem]{Definition}
\newtheorem{example}[theorem]{Example}

\theoremstyle{remark}
\newtheorem{remark}[theorem]{Remark}
\numberwithin{equation}{section}

\newcommand{\Real}{{\mathbb R}}

\newcommand{\N}{{\mathbb N}}
\newcommand{\eps}{\varepsilon}

\newcommand{\x}{\mathbf{x}}
\newcommand{\y}{\mathbf{y}}

\newcommand {\hide}[1]{}

\begin{document}
\title[approximation by compact families]{
Approximation of definable sets by compact families, and upper bounds on homotopy and
homology
}
\author{Andrei Gabrielov}
\address{Department of Mathematics,
Purdue University, West Lafayette, IN 47907, USA}
\email{agabriel@math.purdue.edu}
\author{Nicolai Vorobjov}
\address{
Department of Computer Science, University of Bath, Bath
BA2 7AY, England, UK}
\email{nnv@cs.bath.ac.uk}

\begin{abstract}
We prove new upper bounds on homotopy and homology groups of o-minimal sets
in terms of their approximations by compact o-minimal sets.
In particular, we improve the known upper bounds on Betti numbers of
semialgebraic sets defined by quantifier-free formulae,
and obtain for the first time a singly exponential bound on Betti numbers
of sub-Pfaffian sets.
\end{abstract}
\maketitle

\section*{Introduction}

We study upper bounds on topological complexity of sets
definable in o-minimal structures over the reals.
The fundamental case of algebraic sets in $\Real^n$ was first considered around 1950 by
Petrovskii and Oleinik \cite{O, PO}, and then in 1960s by Milnor \cite{Milnor} and
Thom \cite{Thom}.
They gave explicit upper bounds on total Betti numbers in terms of degrees and numbers of
variables of the defining polynomials.

There are two natural approaches to generalizing and expanding these results.
First, noticing that not much of algebraic geometry is used in the proofs,
one can try to obtain the similar upper bounds for polynomials with the
``description complexity'' measure different from the degree, and for non-algebraic functions,
such as Khovanskii's fewnomials and Pfaffian functions \cite{Khov}.
A bound for algebraic sets defined by quadratic polynomials was proved in \cite{Barvinok}.

Second, the bounds can be expanded to semialgebraic and semi-Pfaffian sets defined by
formulae more general  than just conjunctions of equations.
Basu \cite{Basu99} proved the tight upper bound on Betti numbers in the case of
semialgebraic sets defined by conjunctions and disjunctions of non-strict inequalities.
The proof can easily be extended to special classes of non-algebraic
functions.
For fewnomials and Pfaffian functions, this was done by Zell \cite{Zell}.
For quadratic polynomials an upper bound was proved in \cite{BPasR}.
The principal difficulty arises when neither the set itself nor its complement is locally closed.

Until recently, the best available upper bound for the Betti numbers of a semialgebraic set
defined by an arbitrary Boolean combination of equations and inequalities remained
doubly exponential in the number of variables.
The first singly exponential  upper bound was obtained by the authors in \cite{GV05}
based on a construction which replaces a given
semialgebraic set by a homotopy equivalent compact semialgebraic set.
This construction extends to semi-Pfaffian sets and, more generally, to the
sets defined by Boolean combinations of equations and inequalities between continuous
functions definable in an o-minimal structure over $\Real$.
It cannot be applied to the sets defined by formulae with quantifiers, such
as sub-Pfaffian sets, but can be used in
conjunction with effective quantifier elimination in the semialgebraic situation.

In \cite{GVZ} we suggested a spectral sequence converging to the homology of the projection
of an o-minimal set under the closed continuous surjective definable map.
It gives an upper bound on Betti number of the projection which, in the semialgebraic case, is
better than the one based on quantifier elimination.
The requirement for the map to be closed can be relaxed but not completely removed, which
left the upper bound problem unresolved in the general Pfaffian case, where quantifier
elimination is not applicable.

In this paper we suggest a new construction approximating a large class of definable sets,
including the sets defined by
arbitrary Boolean combinations of equations and inequalities, by compact sets.
The construction is applicable to images of such sets under a large class of definable maps,
e.g., projections.
Based on this construction
we refine the results from \cite{GV05, GVZ}, and prove similar upper
bounds, individual for different Betti numbers, for
images under arbitrary continuous definable maps.

In the semialgebraic case the bound from \cite{GV05} is squaring the number of different
polynomials occurring in the formula, while
the bounds proved in this paper multiply the number of polynomials by a typically smaller
coefficient that does not exceed the dimension.
This is especially relevant for applications to problems of
subspace arrangements, robotics and visualization, where the dimension and degrees
usually remain small, while the number of polynomials is very large.
Applied to projections, the bounds
are stronger than the ones obtained by the effective quantifier elimination.

In the non-algebraic case, for the first time the bounds, singly exponential in the
number of variables,
are obtained for projections of semi-Pfaffian sets, as well as projections
of sets defined by Boolean formulae  with polynomials from
special classes.

\subsection*{Notations}
In this paper we use the following (standard) notations.
For a topological space $X$,
$H_i(X)$ is its singular homology group with coefficients in an Abelian group,
$\pi_i(X)$ is the homotopy group (provided that $X$ is connected),
the symbol $\simeq$ denotes the homotopy equivalence, and the symbol $\cong$ stands for
the group isomorphism.
If $Y \subset X$, then $\overline Y$ denotes its closure in $X$.

\section{Main result}\label{sec:main}

In what follows we fix an o-minimal structure over $\Real$ and consider sets,
families of sets, maps, etc., {\em definable} in this structure.

\begin{definition}\label{def:S_delta}
Let $G$ be a definable compact set.
Consider a definable family $\{ S_\delta \}_{\delta >0}$ of
compact subsets of $G$, such that for all $\delta', \delta >0$,
if $\delta' > \delta$, then $S_{\delta'} \subset S_{\delta}$.
Denote $S := \bigcup_{\delta >0} S_{\delta}$.

For each $\delta >0$, let $\{ S_{\delta, \eps} \}_{\delta, \eps >0}$ be a definable family
of compact subsets of $G$ such that:
\begin{itemize}
\item[(i)]
for all $\eps, \eps' \in (0,1)$, if $\eps' > \eps$, then
$S_{ \delta, \eps} \subset S_{ \delta, \eps'}$;
\item[(ii)]
$S_{\delta}= \bigcap_{\eps >0} S_{\delta, \eps}$;
\item[(iii)]
for all $\delta' >0$ sufficiently smaller than
$\delta$, and for all $\eps' >0$, there exists an open in $G$ set $U \subset G$
such that $S_{\delta} \subset U \subset S_{\delta' , \eps'}$.
\end{itemize}
We say that $S$ is {\em represented} by the families $\{ S_\delta \}_{\delta >0}$ and
$\{ S_{\delta, \eps} \}_{\delta, \eps >0}$ in $G$.
\end{definition}

Let $S'$ (respectively, $S''$) be represented by $\{ S'_\delta \}_{\delta >0}$ and
$\{ S'_{\delta, \eps} \}_{\delta, \eps >0}$ (respectively, by $\{ S''_\delta \}_{\delta >0}$
and $\{ S''_{\delta, \eps} \}_{\delta, \eps >0}$) in $G$.

\begin{lemma}\label{le:bool}
$S' \cap S''$ is represented by the families $\{ S'_\delta \cap S''_\delta \}_{\delta >0}$ and
$\{ S'_{\delta, \eps} \cap S''_{\delta, \eps} \}_{\delta, \eps >0}$ in $G$, while
$S' \cup S''$ is represented by $\{ S'_\delta \cup S''_\delta \}_{\delta >0}$ and
$\{ S'_{\delta, \eps} \cup S''_{\delta, \eps} \}_{\delta, \eps >0}$ in $G$.
\end{lemma}

\begin{proof}
Straightforward checking of Definition~\ref{def:S_delta}.
\end{proof}

Let $S$ be represented by $\{ S_\delta \}_{\delta >0}$ and
$\{ S_{\delta, \eps} \}_{\delta, \eps >0}$ in $G$, and let $F:\> D \to H$ be a continuous
definable map, where $D$ and $H$ are definable, $S \subset D \subset G$, and $H$ is compact.

\begin{lemma}\label{le:map_F}
Let $D$ be open in $G$, and $F$ be an open map.
Then $F(S)$ is represented by families $\{ F(S_{\delta}) \}_{\delta >0}$ and
$\{ F(S_{\delta, \eps}) \}_{\delta, \eps >0}$ in $H$.
\end{lemma}

\begin{proof}
Straightforward checking of Definition~\ref{def:S_delta} (openness is required for
(iii) to hold).
\end{proof}

Consider projections $\rho_1: \> G \times H \to G$ and $\rho_2:\> G \times H \to H$.
Let $\Gamma \subset G \times H$ be the graph of $F$.
Suppose that $\Gamma$ is represented by families $\{ \Gamma_\delta \}_{\delta >0}$ and
$\{ \Gamma_{\delta, \eps} \}_{\delta, \eps >0}$ in $G \times H$.

\begin{lemma}
The set $F(S)$ is represented by the families
$$\{ \rho_2(\rho_{1}^{-1} (S_\delta) \cap \Gamma_\delta) \}_{\delta >0}\quad \text{and}\quad
\{ \rho_2(\rho_{1}^{-1} (S_{\delta, \eps}) \cap \Gamma_{\delta ,\eps}) \}_{\delta, \eps >0}$$
in $H$.
\end{lemma}

\begin{proof}
The set $\rho_{1}^{-1}(S)$ is represented by the families
$$\{ \rho_{1}^{-1} (S_\delta) \cap \Gamma_\delta \}_{\delta >0}\quad \text{and}\quad
\{ \rho_{1}^{-1} (S_{\delta, \eps}) \cap \Gamma_{\delta ,\eps} \}_{\delta, \eps >0}$$
in $G \times H$, and the projection $\rho_2$ satisfies Lemma~\ref{le:map_F}.
\end{proof}

Along with this general case we will be considering the following important particular cases.

Let $S= \{ \x|\> {\mathcal F}(\x)\} \subset \Real^n$ be a bounded definable set of points
satisfying a Boolean combination ${\mathcal F}$
of equations of the kind $h(\x)=0$ and inequalities of the kind $h(\x)>0$,
where $h:\> \Real^n \to \Real$ are continuous definable functions (e.g., polynomials).
As $G$ take a closed ball of a sufficiently large radius centered at 0.
We now define the representing families $\{ S_{\delta} \}$ and
$\{ S_{\delta, \eps} \}$.

\begin{definition}
For a given finite set $\{ h_1, \ldots , h_k \}$  of functions $h_i:\ \Real^n \to \Real$ define
its {\em sign set} as a non-empty subset in $\Real^n$ of the kind
$${h_{i_1} = \cdots = h_{i_{k_1}} = 0, h_{i_{k_1+1}} > 0, \ldots , h_{i_{k_2}} > 0,
h_{i_{k_2+1}} < 0, \ldots , h_{i_k} < 0},$$
where $i_1, \ldots , i_{k_1}, \ldots , i_{k_2}, \ldots , i_k$ is a permutation of
$1, \ldots , k$.
\end{definition}

Let now $\{ h_1, \ldots , h_k \}$ be the set of all functions in the Boolean formula defining
$S$.
Then $S$ is a disjoint union of some sign sets of $\{ h_1, \ldots , h_k \}$.
The set $S_{\delta}$ is the result of the replacement independently in each sign set of
all inequalities $h>0$ and $h<0$ by $h \ge \delta$ and $h \le -\delta$ respectively.
The set $S_{\delta, \eps}$ is obtained by replacing independently in each sign set
all expressions $h>0$, $h<0$ and $h=0$ by $h \ge \delta$, $h \le -\delta$ and
$-\eps \le h \le \eps$, respectively.
According to Lemma~\ref{le:bool}, the set $S$, being the union of sign sets, is represented by
families $\{ S_{\delta} \}$ and $\{ S_{\delta, \eps} \}$ in $G$.

\begin{example}
Let the closed quadrant $S$ be defined as the union of sign sets
$\{ x>0,\ y>0 \} \cup \{ x>0, \ y=0 \} \cup \{x=0,\ y>0 \} \cup \{ x=y=0 \}$.
Fig.~\ref{fig1} shows the corresponding set $S_{\delta, \eps}$ for $\eps < \delta$.
\end{example}

Now suppose that the set $S \subset \Real^n$, defined as above by a Boolean formula
$\mathcal F$, is not necessarily bounded.
In this case as $G$ take the definable one-point (Alexandrov)
compactification of $\Real^n$.
Note that each function $h$ is continuous in $G \setminus \{ \infty \}$.
Define sets $S_\delta$ and $S_{\delta, \eps}$ as in the bounded case, replacing equations and
inequalities independently in each sign set of $\{ h_1, \ldots , h_k \}$,
and then taking the conjunction of the resulting formula with $|\x|^2 \le 1/\delta$.
Again, $S$ is represented by $\{ S_{\delta} \}$ and $\{ S_{\delta, \eps} \}$ in $G$, and
in the sequel we will refer to this instance as the {\em constructible case}.

\begin{figure}
\epsfxsize=3.0in
\centerline{\epsffile{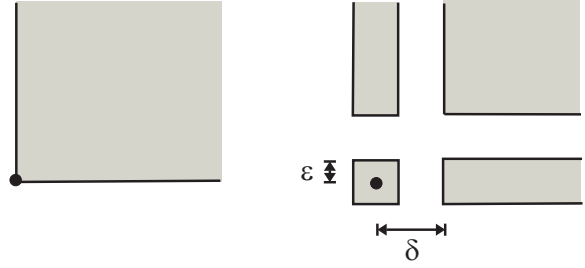}}
\caption{The set $S_{\delta,\eps}$ (right) for the closed
quadrant $S$ (left).}
\label{fig1}
\end{figure}

\begin{definition}
Let ${\mathcal P}:={\mathcal P}(\eps_0, \ldots , \eps_\ell)$ be a
predicate (property) over $(0,1)^{\ell+1}$.
We say that the property ${\mathcal P}$ holds for
$$
0<\eps_0 \ll\eps_1 \ll \cdots \ll\eps_\ell \ll 1,
$$
if there exist definable functions
$f_k:\> (0,1)^{\ell-k} \to (0,1)$, $k=0, \ldots ,\ell$ (with $f_\ell$ being a positive
constant) such that ${\mathcal P}$
holds for any sequence $\eps_0, \ldots , \eps_\ell$ satisfying
$$0 < \eps_k < f_k(\eps_{k+1}, \ldots , \eps_\ell)\quad \text{for}\quad
k=0, \ldots, \ell.$$
\end{definition}

Now we return to the general case in the Definition~\ref{def:S_delta}, which we will refer to,
in what follows, as the {\em definable case}.

\begin{definition}\label{def:telescope}
For a sequence $\eps_0 , \delta_0 ,\eps_1 , \delta_1 , \ldots ,\eps_m , \delta_m$,
where $m \ge 0$, introduce the compact set
$$T(S):=S_{\delta_0,\eps_0}\cup S_{\delta_1,\eps_1} \cup \cdots \cup S_{\delta_m,\eps_m}.$$
\end{definition}

From Definition~\ref{def:S_delta} it is easy to see that for any $m \ge 0$, and for
\begin{equation}\label{eq:ed}
0<\eps_0 \ll\delta_0\ll\eps_1 \ll \delta_1 \ll \cdots \ll\eps_m \ll \delta_m \ll 1,
\end{equation}
there is a surjective map $C:\> {\bf T} \to {\bf S}$
from the finite set ${\bf T}$ of all connected components of $T(S)$ onto the set ${\bf S}$
of all connected components of $S$, such that for any $S' \in {\bf S}$,
$$\bigcup_{T' \in C^{-1}(S')}T'= T(S').$$

\begin{lemma}\label{le:components}
If $m>0$ then $C$ is bijective.
\end{lemma}

\begin{proof}
Let $S$ be connected and $m>0$.
We prove that $T(S)$ is connected.
Let $\x, \y \in S_{\delta_i, \eps_i} \subset T(S)$.
Let $\x_{\eps},\> \y_{\eps}$ be a definable connected curves such that
$\x_{\eps_i} =\x$, $\x_0:=\lim_{\eps\searrow 0} \x_{\eps}\in S_{\delta_i}$,
$\y_{\eps_i} =\y$, and $\y_0:=\lim_{\eps\searrow 0} \y_{\eps}\in S_{\delta_i}$.
Let $\Gamma \subset S$ be a connected compact definable curve containing
$\x_0$ and $\y_0$.
Then $\Gamma$ is represented by the families
$\{ S_{\delta} \cap \Gamma \}$ and $\{ S_{\delta, \eps} \cap \Gamma \}$ in $\Gamma$, hence
$T(\Gamma) \subset T(S)$.
It is easy to see that, under the condition $m>0$, the one-dimensional $T(\Gamma)$ is connected.
It follows that $\x$ and $\y$ belong to a connected definable curve in $T(S)$.
\end{proof}

In what follows we denote $T:= T(S)$, and let $m>0$.
We assume that $S$ is connected in order to make the homotopy groups $\pi_k(S)$ and
$\pi_k(T)$ independent of a base point.

\begin{theorem}\label{th:main}
\begin{itemize}
\item[(i)]
For (\ref{eq:ed}) and every $1 \le k \le m$, there are epimorphisms
$$\psi_k:\> \pi_k(T) \to \pi_k(S),$$
$$\varphi_k:\> H_k(T) \to H_k(S),$$
in particular, ${\rm rank}\ H_k(S) \le {\rm rank}\ H_k(T)$.
\item[(ii)]
In the constructible case, for (\ref{eq:ed}) and every $1 \le k \le m-1$, $\psi_k$
and $\varphi_k$ are
isomorphisms, in particular, ${\rm rank}\ H_k(S) = {\rm rank}\ H_k(T)$.
Moreover, if $m \ge \dim (S)$, then $T \simeq S$.
\end{itemize}
\end{theorem}

The plan of the proof of Theorem~\ref{th:main} is as follows.
We consider a simplicial complex $R$ in $\Real^n$ such that it is a triangulation of $G$,
and $S$ is a union of some open simplices of $R$.
For any sequence $\eps_0 , \delta_0 ,\eps_1 , \delta_1 , \ldots ,\eps_m , \delta_m$
we construct a subset $V$ of the complex $R$, which is a combinatorial analogy of $T$, and
prove that there are isomorphisms of $k$-homotopy groups of $V$ and $S$ for $k \le m-1$
and an epimorphism for $k=m$.
We prove the same for homology groups.
We then show that for (\ref{eq:ed}) there are epimorphisms
$\pi_k(T) \to \pi_k(V)$ and $H_k(T) \to H_k(V)$ for $k \le m$.
We prove that if the pair $(R, S_{\delta})$ satisfies a certain ``separability'' property
(Definition~\ref{def:separ}), then $V \simeq T$.
In particular, in the constructible case $(R, S_{\delta})$ is always separable.
This completes the proof.

\begin{remark}
We conjecture that in the {\em definable} case the statement (ii) of Theorem~\ref{th:main}
is also true, i.e., for (\ref{eq:ed}) and every $1 \le k \le m-1$, the homomorphisms $\psi_k$,
$\varphi_k$ are isomorphisms, and $T \simeq S$ when $m \ge \dim (S)$.
\end{remark}

\section{Topological background}

In this section we formulate some topological definitions and statements which we will use in
further proofs.

Recall that a continuous map between topological spaces $f:\> X \to Y$ is called a {\em weak
homotopy equivalence} if for every $j >0$ the induced homomorphism of homotopy groups
$f_{\# j}:\> \pi_j(X) \to \pi_j(Y)$ is an isomorphism.

\begin{theorem}[Whitehead Theorem on weak homotopy equivalence, \cite{Spanier},
7.6.24]\label{th:Whitehead1}
A map between connected CW-complexes is a weak homotopy equivalence iff it is a
homotopy equivalence.
\end{theorem}

Let $f:\ X \to Y$ be a continuous map between path-connected topological spaces.

\begin{theorem}[Whitehead Theorem, \cite{Spanier}, 7.5.9]\label{th:Whitehead2}
If there is $k >0$ such that the induced homomorphism of homotopy groups
$f_{\# j}:\> \pi_j(X) \to \pi_j(Y)$
is an isomorphism for $j<k$ and an epimorphism for $j=k$, then the induced homomorphism of
homology groups $f_{\ast j}:\> H_j(X) \to H_j(Y)$ is an isomorphism for $j<k$ and an
epimorphism for $j=k$.
\end{theorem}

\begin{definition}[\cite{Bjorner}]
A map $f:\> P \to Q$, where $P$ and $Q$ are posets with order relations $\preceq_P$ and
$\preceq_Q$ respectively, is called {\em poset map} if, for $\x, \y \in P$,
$\x \preceq_P \y$ implies $f(\x) \preceq_Q f(\y)$.
With a poset $P$ is associated the simplicial complex $\Delta(P)$, called {\em order
complex}, whose simplices are {\em chains} (totally ordered subsets) of $P$.
Each poset map $f$ induces the simplicial map $f:\> \Delta(P) \to \Delta(Q)$.
\end{definition}

\begin{theorem}[\cite{Bjorner}, Th. 2]\label{th:Bjorner2}
Let $P$ and $Q$ be connected posets and $f:\> P \to Q$ a poset map.
Suppose that the fibre $f^{-1}(\Delta(Q_{\preceq q}))$ is $k$-connected for all $q \in Q$.
Then the induced homomorphism $f_{\# j}:\> \pi_j(\Delta(P)) \to \pi_j(\Delta (Q))$ is an
isomorphism for all $j \le k$ and an epimorphism for $j=k+1$.
\end{theorem}

\begin{remark}\label{re:Bjorner2}
In the formulation and proof of this theorem in \cite{Bjorner} the statement that
$f_{\# k+1}$ is an epimorphism, is missing.
Here is how it follows from the proof of Theorem~2 in \cite{Bjorner}.
In the proof, a map $g:\> \Delta^{(k+1)}(Q) \to \Delta (P)$ is defined, where
$\Delta^{(k+1)}(Q)$ is the $(k+1)$-dimensional skeleton of $\Delta(Q)$, such that
$f \circ g:\> \Delta^{(k+1)}(Q) \to \Delta (Q)$ is homotopic to the identity map $id$.
Then the induced homomorphism
$$f_{\# k+1} \circ g_{\# k+1}=(f \circ g)_{\# k+1}= id_{\# k+1}:\>
\pi_{k+1}(\Delta^{(k+1)} (Q)) \to \pi_{k+1}(\Delta (Q))$$
is an epimorphism, since any map of a $j$-dimensional sphere to $\Delta (Q)$ is homotopic
to a map of the sphere  to $\Delta^{(j)}(Q)$.
It follows that $f_{\# k+1}$ is also an epimorphism.
\end{remark}

\begin{corollary}[Vietoris-Begle Theorem]\label{cor:VB}
Let $X$ and $Y$ be connected simplicial complexes and $f:\> X \to Y$ a simplicial map.
\begin{itemize}
\item[(i)]
If the fibre $f^{-1}(B)$ is $k$-connected for every closed simplex $B$ in $Y$,
then the induced homomorphism $f_{\# j}:\> \pi_j(X) \to \pi_j(Y)$ is an isomorphism for
all $j \le k$ and an epimorphism for $j=k+1$.
\item[(ii)]
If the fibre $f^{-1}(B)$ is contractible, then $X \simeq Y$.
\end{itemize}
\end{corollary}

\begin{proof}
(i)\quad
Consider barycentric subdivisions $\widehat X$ and $\widehat Y$ of complexes $X$ and $Y$
respectively.
Note that $\widehat X= \Delta (P)$ and $\widehat Y= \Delta (Q)$ where $P$ and $Q$ are
{\em simplex posets} of $X$ and $Y$ respectively (i.e., closed simplices ordered by
containment).
For a closed simplex $B \in Q$ the subcomplex $\Delta (Q_{\preceq B})$ of $\widehat Y$ is
the union of all simplices of the barycentric subdivision of $B$.
Now (i) follows from Theorem~\ref{th:Bjorner2}.

(ii)\quad Since the fibre $f^{-1}(B)$ is contractible, according to (i), the induced
homomorphisms $f_{\# j}$ are isomorphisms for all $j > 0$, hence, by Whitehead
theorem on weak homotopy equivalence (Theorem~\ref{th:Whitehead1}),
$f$ induces the homotopy equivalence $X \simeq Y$.
\end{proof}

\begin{definition}\label{def:nerve}
The {\em nerve} of a family $\{ X_i \}_{i \in I}$ of sets is the (abstract) simplicial complex
${\mathcal N}$ defined on the vertex set $I$ so that a simplex $\sigma \subset I$ is
in ${\mathcal N}$ iff $\bigcap_{i \in \sigma} X_i \neq \emptyset$.
\end{definition}

Let $X$ be a connected regular CW-complex and $\{ X_i \}_{i \in I}$ be a family of its
subcomplexes such that $X= \bigcup_{i \in I} X_i$.
Let $|{\mathcal N}|$ denote the geometric realization of the nerve ${\mathcal N}$ of
$\{ X_i \}_{i \in I}$.

\begin{theorem}[Nerve Theorem, \cite{Bjorner}, Th. 6]\label{th:nerve}
\begin{itemize}
\item[(i)]
If every nonempty finite intersection $X_{i_1} \cap \cdots \cap X_{i_t}$, $t \ge 1$, is
$(k-t+1)$-connected, then there is a map $f:\> X \to |{\mathcal N}|$ such that the induced
homomorphism $f_{\# j}:\> \pi_j(X) \to \pi_j(|{\mathcal N}|)$ is an
isomorphism for all $j \le k$ and an epimorphism for $j=k+1$.
\item[(ii)]
If every nonempty finite intersection $X_{i_1} \cap \cdots \cap X_{i_t}$, $t \ge 1$, is
contractible, then $X \simeq |{\mathcal N}|$.
\end{itemize}
\end{theorem}

\begin{remark}
As with Theorem~\ref{th:Bjorner2}, in the formulation and proof of this theorem in
\cite{Bjorner} the statement that $f_{\# k+1}$ is an epimorphism, is missing.
This statement follows from the proof of Theorem~6 in \cite{Bjorner} by the same
argument as described in Remark~\ref{re:Bjorner2}.
\end{remark}

\begin{remark}\label{re:simp_nerve}
Let $X$ be a triangulated set, $\{ X_i \}_{i \in I}$ be a family
of all of its simplices, and the nerve ${\mathcal N}$ is defined on the index set $I$
so that a simplex $\sigma \subset I$ is in ${\mathcal N}$ iff the family
$\{ X_i \}_{i \in \sigma}$, after a the suitable ordering, forms a $|\sigma|$-flag
(see Definition~\ref{def:flag} below).
It is easy to see that this version of the nerve can be reduced to the one in the
Definition~\ref{def:nerve}, so that the Theorem~\ref{th:nerve} holds true.
\end{remark}

\begin{definition}
For two continuous maps $f_1: X_1 \to Y$ and $f_2: X_2 \to Y,$ the {\em fibred product}
is defined as
$$X_1 \times_Y X_2 :=\{ (\textbf{x}_1, \textbf{x}_2) \in X_1 \times X_2|\>
f_1(\textbf{x}_1)=f_2(\textbf{x}_2) \}.$$
\end{definition}

\begin{theorem}[\cite{GVZ}, Th. 1]\label{th:spectral}
Let $f: X \to Y$ be a continuous closed surjective o-minimal map.
Then there is a spectral sequence $E^{r}_{p,q}$
converging to $H_{\ast}(Y)$ with $E_{p,q}^{1}=H_q(W_p)$, where
$W_p:= \underbrace{X\times_Y\cdots\times_Y X}_{p+1\> \>{\rm times}}$.
\end{theorem}

\begin{corollary}\label{cor:spectral}
For $f: X \to Y$ as in Theorem~\ref{th:spectral} and for any $k \ge 0$
$${\rm b}_k(Y) \le \sum_{p+q=k} {\rm b}_q(W_p),$$
where ${\rm b}_k := {\rm rank}\> H_k$ is the $k$-th Betti number.
\end{corollary}

\section{Simplicial construction}\label{sec:simp_constr}

Since $G$ and $S$ are definable, they are triangulable (\cite{Coste}, Th. 4.4), i.e.,
there exists a
finite simplicial complex $R= \{ \Delta_j \}$ and a definable homeomorphism
$\Phi:\> |R| \to G$, where $|R|$ is the geometric realization of $R$, such that
$S$ is a union of images under $\Phi$ of some simplices of $R$.
By a {\em simplex} we always mean an {\em open simplex}.
If $\Delta$ is a simplex, then $\overline \Delta$ denotes its closure.
In what follows we will ignore the distinction between simplices of $|R|$
and their images in $G$.

\begin{definition}\label{def:flag}
For a simplex $\Delta$ of $S$, its {\em subsimplex} is a simplex $\Delta' \neq \Delta$ such
that $\Delta' \subset \overline \Delta$.
A {\em $k$-flag} of simplices of $R$ is a sequence $\Delta_{i_0}, \ldots ,\Delta_{i_k}$
such that $\Delta_{i_\nu}$ is a subsimplex of $\Delta_{i_{\nu-1}}$ for $\nu = 1, \ldots ,k$.
\end{definition}

\begin{definition}\label{def:marked}
The set $S$ is {\em marked} if for every pair $(\Delta', \Delta)$ of simplices of $S$, such
that $\Delta'$ is a subsimplex of $\Delta$, the simplex $\Delta'$ is classified as  either
{\em hard} or {\em soft} subsimplex of $\Delta.$
If $\Delta'$ is not in S, it is always soft.
\end{definition}

In what follows we assume that $S$ is marked.

Let $\widehat R$ be the barycentric subdivision of $R$.
Then each vertex $v_j$ of $\widehat R$ is the center of a simplex $\Delta_j$ of $R$.
Let $B=B(j_0, \ldots ,j_k)$ be a $k$-simplex of $\widehat R$ having vertices
$v_{j_0}, \ldots ,v_{j_k}$.
Assume that the vertices of $B$ are ordered so that $\dim \Delta_{j_0} > \cdots >
\dim \Delta_{j_k}$.
Then $B$ corresponds to a $k$-flag
$\Delta_{j_0}, \ldots ,\Delta_{j_k}$ of simplices of $R$.
Let $\widehat S$ be the set of simplices of $\widehat R$ which belong to $S$.
Then $S$ is the union of all simplices of $\widehat S$.

\begin{definition}
The {\em core} $C(B)$ of a simplex $B=B(j_0, \ldots ,j_k)$ of $\widehat S$ is the maximal
subset $\{ j_0, \ldots ,j_p \}$
of the set $\{ j_0, \ldots ,j_k \}$ such that $\Delta_{j_\nu}$ is a {\em hard}
subsimplex of $\Delta_{j_\mu}$ for all $\mu < \nu \le p$.
Note that $j_0$ is always in $C(B)$, in particular, $C(B) \neq \emptyset$.
Assume that for a simplex $B$ not in $\widehat S$, the core $C(B)$ is empty.
\end{definition}

\begin{lemma}\label{le:core}
Let $B=B(i_0, \ldots ,i_k)$ be a simplex in $\widehat S$, and $K=K(j_0, \ldots ,j_{\ell})$
be a simplex in $\widehat R$, with $B \subset \overline K$, i.e., $I=\{ i_0, \ldots ,i_k \}
\subset J= \{ j_0, \ldots ,j_{\ell} \}$.
Then $I \setminus C(B) \subset J \setminus C(K)$.

\end{lemma}

\begin{proof}
Straightforward consequence of the definitions.
\end{proof}

\begin{definition}
For two simplices $B$ and $B'$ of $\widehat S$, let $B' \succeq B$ if either
$B'$ is a subsimplex of $B$ (reverse inclusion) and $C(B') \cap C(B) = \emptyset$, or $B'=B$.
If $B' \succeq B$ and $B' \neq B$, then we write $B' \succ B$.
Lemma~\ref{le:core} implies that $\succeq$ is a partial order on the set of all simplices of
$\widehat S$.
The {\em rank} $r(\widehat S)$ of $\widehat S$ is the maximal length $r$ of a chain
$\Delta_0 \succ \cdots \succ \Delta_r$ of simplices in $\widehat S$.
Let $B$ be a simplex in $\widehat S$.
The set $S_B$ of simplices $B' \subset \overline B \cap \widehat S$ is a poset with partial
order induced from $(\widehat S, \succeq)$.
The {\em rank $r(S_B)$ of} $S_B$ is the maximal length of its chain.
\end{definition}

\begin{definition}\label{def:K_B}
Let simplices $B$ and $K$ be as in Lemma~\ref{le:core}.
For $0 < \delta <1$, define
$$B(\delta):= \Bigg\{ \sum_{i_\nu \in I} t_{i_\nu}v_{i_\nu} \in B(i_0, \ldots ,i_k) \Bigg|\>
\sum_{i_\nu \in C(B)} t_{i_\nu}> \delta  \Bigg\}.$$
For $0 < \eps <1$ and $0 < \delta <1$, define
$$
K_B(\delta, \eps):=   \Bigg\{ \sum_{j_\nu \in J} t_{j_\nu}v_{j_\nu}
\in K(j_0, \ldots ,j_\ell) \Bigg|\>
\sum_{i_\nu \in C(B)} t_{i_\nu}> \delta,\> \sum_{i_\nu \in I}t_{i_\nu}>1 - \eps,
$$
$$
\forall i_\nu \in I\> \forall j_\mu \in (J \setminus I)\ (t_{i_\nu} > t_{j_\mu}) \Bigg\}.
$$
\end{definition}

\begin{definition}\label{def:V}
Let $B$ be a simplex in $\widehat S$.
Fix some $m \ge 0$ and a sequence
$\eps_0, \delta_0, \eps_1, \delta_1, \ldots ,\eps_m, \delta_m $.
Define $V_B$ as the union of sets $K_{B'}(\delta_i, \eps_i)$ over all simplices
$B' \in S_B$, simplices $K$ of $\widehat R$ such that $B \subset \overline K$,
and $i=0, \ldots, m$.
Define $V$ as the union of the sets $V_B$ over all simplices $B$ of $\widehat S$.
\end{definition}

\section{Topological relations between $V$ and $S$}

\begin{lemma}\label{le:K_B_intersection}
Let $B=B(i_0, \ldots ,i_k)$ be a simplex in $\widehat S$, and $K=K(j_0, \ldots ,j_{\ell})$
be a simplex in $\widehat R$, with $B \subset \overline K$.
Then
$$K_B(\delta, \eps) \cap K_B(\delta', \eps')=K_B(\max \{\delta, \delta' \},
\min \{ \eps, \eps' \}),$$
for all $0 < \delta, \eps, \delta', \eps' <1$.
\end{lemma}

\begin{proof}
Straightforward consequence of the definitions.
\end{proof}

\begin{lemma}\label{le:B_B'}
For any two simplices $B$ and $B'$ of $\widehat S$, a simplex $K $ of $\widehat R$ such that
$B$ and $B'$ are subsimplices of $K$, and all $0 < \delta, \eps, \delta', \eps' <1$,
\begin{itemize}
\item[(i)]
if $K_B(\delta, \eps) \cap K_{B'}(\delta', \eps') \neq \emptyset$, then either
$B \subset \overline {B'}$ or $B' \subset \overline B$;
\item[(ii)]
$K_B(\delta, \eps) \cap K_{B'}(\delta', \eps')$ is convex.
\end{itemize}
\end{lemma}

\begin{proof}
Straightforward consequence of the definitions.
\end{proof}

\begin{lemma}\label{le:Z_K}
Let $K$ be a simplex of $\widehat R$, and let $B_0, \ldots , B_k$ be a flag of simplices of
$\widehat S$, with $B_0  \subset \overline K$.
Then for (\ref{eq:ed})
and a sequence $i_0,j_0, \ldots ,i_k,j_k$ of integers in $\{ 0, 1, \ldots ,m \}$,
the intersection
$$Z_K(i_0,j_0, \ldots ,i_k,j_k):= K_{B_0}(\delta_{i_0}, \eps_{j_0}) \cap \cdots \cap
K_{B_k}(\delta_{i_k}, \eps_{j_k})$$
is non-empty if and only if $B_\mu \succ B_\nu$ implies $j_\mu > i_\nu$
for any $\mu,\ \nu \in \{0, 1, \ldots ,k \}$.
\end{lemma}

\begin{proof}
The necessity of the condition is straightforward.
To show that it is sufficient we will construct a point $v:= \sum t_jv_j$, where the
sum is taken over all vertices $v_j$ of $K$, such that $v \in Z_K(i_0,j_0, \ldots ,i_k,j_k)$.
This will be done in three steps.

(a)\ Define $\ell_\nu$ as the last index in $C(B_\nu)$ (i.e., $v_{\ell_\nu}$ is the center of
the smallest simplex $\Delta_j$ of $R$ such that $j \in C(B_\nu)$).
Set $t_{\ell_\nu}:= \delta_{i_\nu}$.
If $\ell_\nu$ is the same index for several $\nu$, set $t_{\ell_\nu}$to be the maximum of the
corresponding $\delta_{i_\nu}$.

(b)\ Fix a sequence $\gamma_0, \ldots , \gamma_{k+1}$ such that $0< \gamma_0 < \cdots <
\gamma_{k+1} \ll \eps_0$.
For a vertex $v_j$ of $B_{\nu -1}$ which is not one of $v_{\ell_\mu}$ and not a vertex of
$B_\nu$, set $t_j:= \gamma_\nu + \max \delta_{i_\mu}$, where the maximum is taken over all
$\mu$ such that $B_\nu \succ B_\mu$ (or equals 0 if there is no such $\mu$).
For any vertex $v_j$ of $K$ that does not belong to $B_0$, set $t_j:= \gamma_0$.
For a vertex $v_j$ of $B_k$ that is not one of $v_{\ell_\nu}$ set $t_j:= \gamma_{k+1} +
\max \delta_{i_\mu}$, where the maximum is taken over $\mu = 0, \ldots ,k$.

(c)\ For the last vertex $v_\omega$ of $B_k$ set $t_\omega:= 1 - \sum_{v_j}t_j$,
where the sum is taken
over all vertices $v_j$ of $K$ other than $v_\omega$.
If $\omega= \ell_k$, this overrides the setting in (a).
If $\omega \neq \ell_k$, this overrides the setting in (b).

It is easy to check that $v \in Z_K(i_0,j_0, \ldots ,i_k,j_k)$.
\end{proof}

\begin{lemma}\label{le:contractible}
Let $B_0, \ldots , B_k$ be a flag of simplices of $\widehat S$, and
$i_0,j_0, \ldots ,i_k,j_k$ be a sequence of integers in $\{ 0, 1, \ldots ,m \}$.
For (\ref{eq:ed}),
if $B_\mu \succ B_\nu$ implies $j_\mu > i_\nu$ for any $\mu,\ \nu \in \{0, 1, \ldots ,k \}$,
then the set
$$Z(i_0,j_0, \ldots ,i_k,j_k):= \bigcup_K Z_K(i_0,j_0, \ldots ,i_k,j_k),$$
where the union is taken over all simplices $K$ of $\widehat R$ with $B_0 \subset \overline K$,
is an open contractible subset of $G$.
Otherwise $Z(i_0,j_0, \ldots ,i_k,j_k)= \emptyset$.
\end{lemma}

\begin{proof}
By Lemma~\ref{le:Z_K}, $Z(i_0,j_0, \ldots ,i_k,j_k) \neq \emptyset$ if and only if
$B_\mu \succ B_\nu$ implies $j_\mu > i_\nu$ for any $\mu,\ \nu \in \{0, 1, \ldots ,k \}$.
So, suppose that $Z(i_0,j_0, \ldots ,i_k,j_k) \neq \emptyset$, and consider two simplices,
$K$ and $K'$, such that $B_0 \subset \overline {K'} \subset \overline K$.
Then the intersection of the closure of the complement in $K$ of
$Z_K(i_0,j_0, \ldots ,i_k,j_k)$ with $K'$ coincides with the complement in $K'$ of
$Z_{K'}(i_0,j_0, \ldots ,i_k,j_k)$.
Hence, $Z_{K'}(i_0,j_0, \ldots ,i_k,j_k) \cup Z_K(i_0,j_0, \ldots ,i_k,j_k)$ is open in
$\overline K$.
It follows that $Z(i_0,j_0, \ldots ,i_k,j_k)$ is open in $G$, and has a closed covering by
convex sets $\overline {Z_K(i_0,j_0, \ldots ,i_k,j_k)} \cap Z(i_0,j_0, \ldots ,i_k,j_k)$
over all simplices $K$ of $\widehat R$.
This covering has the same nerve as the star of $B_0$ in the complex $\widehat R$, this star
is contractible.
An intersection of any number of elements of the covering of $Z(i_0,j_0, \ldots ,i_k,j_k)$
is convex, and therefore contractible.
By the Nerve Theorem (Theorem~\ref{th:nerve} (ii)), both $Z(i_0,j_0, \ldots ,i_k,j_k)$
and the star are
homotopy equivalent to the geometric realization of the nerve, and hence to one another.
It follows that $Z(i_0,j_0, \ldots ,i_k,j_k)$ is contractible.
\end{proof}

\begin{lemma}\label{le:V_B}
For (\ref{eq:ed}) and for each simplex $B$ in $\widehat S$ and every $m \ge 1$ the set $V_B$
(see Definition~\ref{def:V}) is open in $G$ and $(m-1)$-connected.
\end{lemma}

\begin{proof}
For every simplex $B' \in S_B$ consider the set
$U_{B',i}:= \bigcup_K K_{B'}(\delta_i, \eps_i)$,
where the union is taken over all simplices $K$ of $\widehat R$ with $B \subset \overline K$.
Obviously, the family $\{ U_{B',i}|\> B' \in S_B,\> 0 \le i \le m \}$ is an open covering of
$V_B$.
Let $M_B$ denote the nerve of this covering.
From Lemmas~\ref{le:K_B_intersection}, \ref{le:contractible}, $M_B$ is the simplicial complex
whose $k$-simplices can be identified with all sequences of the kind
$((p_0, i_0), \ldots , (p_k, i_k))$, where $p_\nu$ are indices of the simplices
$B'_{p_\nu} \in S_B$, such that
\begin{itemize}
\item[(a)]
$B'_{p_\nu} \subset \overline{B'_{p_{\nu -1}}}$,
\item[(b)]
$0 \le i_\nu \le m$,
\item[(c)]
if $B'_{p_\mu} \succeq B'_{p_\nu}$ and $\mu > \nu$, then $i_\mu > i_\nu$.
\end{itemize}
By Lemma~\ref{le:contractible} any non-empty intersection of sets $U_{B',i}$ is contractible.
Therefore, due to the Nerve Theorem (Theorem~\ref{th:nerve} (ii)),
$V_B$ is homotopy equivalent to $M_B$, and
in order to prove that $V_B$ is $(m-1)$-connected it is sufficient
to show that $M_B$ is an $(m-1)$-connected simplicial complex.
This follows from Proposition~\ref{prop:combinatorics} below.
\end{proof}

Let $\succeq$ be a poset on $\{ 0, \ldots ,N \}$ such that if $p \succeq q$ and $p \neq q$,
then $p>q$.
For each $p \in \{0, \ldots, N \}$, let $r(p)$ be the maximal length of a poset chain with the
maximal element $p$ (i.e., the rank of the order ideal generated by $p$).
Let $m_0, \ldots ,m_N$ be nonnegative integers.
Let $M(m_0, \ldots, m_N)$ be the simplicial complex containing all $k$-simplices
$((p_0,i_0), \ldots ,(p_k,i_k))$ such that
\begin{itemize}
\item[(a)]
$p_\nu \in \{0, \ldots, N \}$, $p_0 \le \cdots \le p_k$,
\item[(b)]
$i_\nu \in \{ 0, \ldots ,m_\nu \}$,
\item[(c)]
if $p_\mu \succeq p_\nu$ and $\mu > \nu$, then $i_\mu > i_\nu$.
\end{itemize}
Let $m:= \min \{m_1, \ldots , m_N \}$.

An example of the complex $M(2,2)$ with $1 \succ 0$ is shown on Fig~\ref{fig2}.

\begin{figure}
\epsfxsize=2.5in
\centerline{\epsffile{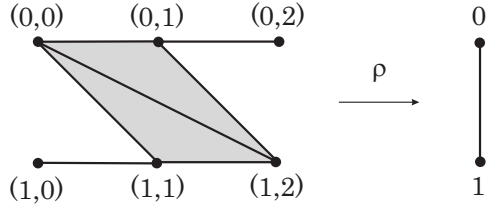}}
\caption{The complex $M(2,2)$ with $1 \succ 0$.}
\label{fig2}
\end{figure}

\begin{proposition}\label{prop:combinatorics}
The simplicial complex $M(m_0, \ldots ,m_N)$ is $(m-1)$-connected.
\end{proposition}

\begin{proof}
Let $\Delta^N$ be the $N$-simplex, and $\overline{\Delta^N(m)}$ be the $m$-dimensional
skeleton of its closure.
There is a natural simplicial map
$$\begin{matrix}
\rho:\> M(m_0, \ldots ,m_N) \to \overline{\Delta^N},\cr
(p,i) \mapsto p.
\end{matrix}$$
It is easy to see that $\overline{\Delta^N(m)} \subset \rho (M(m_0, \ldots ,m_N))$, hence
$\rho (M(m_0, \ldots ,m_N))$ is $(m-1)$-connected.

Consider any face $\Delta^L$ of $\Delta^N$, $L \le N$, which has nonempty pre-image under $\rho$.
Without loss of generality assume that its vertices are $0, \ldots ,L$.
Let $M(m_0, \ldots ,m_L)$ be the simplicial complex defined over the poset on
$\{0, \ldots ,L \}$ induced by $\succeq$.
We prove inductively on $L$ that for any point $\x \in \Delta^L$ the fibre $\rho^{-1}(\x)$ is
contractible.
The Proposition then follows from Vietoris-Begle Theorem (Corollary~\ref{cor:VB} (ii)).
The base of induction, for $L=0$, is obvious.
Assume that the statement is true for $L-1$.
For any simplex $K=((p_0,i_0), \ldots , (p_k,i_k))$ of $M(m_0, \ldots, m_L)$ that projects
surjectively onto $\Delta^L$, if $p_\nu=L$ then $i_\nu \ge r(L)$.
Let $s=i_{\ell}$ be the minimal of these $i_\nu$ in $K$, so that $p_\nu <L$ for $\nu < \ell$,
while $p_\ell=L$.
Then $((p_0,i_0), \ldots ,(p_{\ell -1},i_{\ell -1}))$ is a simplex of the simplicial complex
$M(s):=M(m'_0, \ldots ,m'_{L-1})$, defined over the poset on $\{0, \ldots ,L-1 \}$ induced by
$\succeq$, where $m'_p:= \min \{m_p, r(p)+s -r(L) \}$ if $L \succeq p$, and $m'_p:=m_p$ if $L$
is incomparable with $p$.
It follows that $K$ is a simplex of the join of $M(s)$ and $\overline{\Delta^{m_L-s}}$,
where $\Delta^{m_L-s}$ is the simplex with vertices $s, \ldots ,m_L$.
Since the complex $M(s)$ is contractible due to the induction hypothesis, its join with
$\overline{\Delta^{m_L-s}}$ has a contractible fibre over any $\x \in \Delta^L$.
The fibre over $\x$ of $M(m_0, \ldots, m_L)$ is the union of these contractible fibres for
$s=r(N), \ldots ,m_L$.
The intersection of any number of these fibres is nonempty and contractible, being a fibre of
the join of $M_{\min}$ and $\overline{\Delta^{m_L-s_{\max}}}$.
Due to the Nerve Theorem (Theorem~\ref{th:nerve} (ii)), their union is homotopy equivalent to
its nerve, a simplex, and thus is contractible.
\end{proof}

\begin{corollary}\label{cor:combinatorics}
In the definition of the simplicial complex $M(m_0, \ldots, m_N)$ assume additionally
that $m_j \ge r(j)$ for every $j=0, \ldots ,N$.
Then  $M(m_0, \ldots ,m_N)$ is contractible.
\end{corollary}

\begin{proof}
The condition $m_j \ge r(j)$ guarantees that the map $\rho$ is surjective,
hence $\rho (M(m_0, \ldots ,m_N))$ is contractible.
\end{proof}

\begin{theorem}\label{th:VS}
For (\ref{eq:ed})
there are homomorphisms $\chi_k:\> H_k(V) \to H_k(S)$ and $\tau_k:\> \pi_k(V) \to \pi_k(S)$
such that $\chi_k$, $\tau_k$ are isomorphisms
for every $k \le m-1$, and $\chi_m$, $\tau_m$ are epimorphisms.
Moreover, if $m \ge \dim (S)$, then $V \simeq S$.
\end{theorem}

\begin{proof}
Due to Lemma~\ref{le:B_B'}(i), for any three simplices $B_0$, $B_1$, $B_2$ in $\widehat S$,
the equality
$\overline B_0= \overline B_1 \cap \overline B_2$ is equivalent to
$V_{B_0}=V_{B_1} \cap V_{B_2}$.
Hence, a nonempty intersection of any number of sets $V_B$ is a set of the same type, and
therefore is $(m-1)$-connected.
Moreover, there is an isomorphism $\xi:\> |{\mathcal N}_V| \to |{\mathcal N}_{\widehat S}|$
between the geometric realization of the nerve ${\mathcal N}_{\widehat S}$ of the covering
of the simplicial complex $\widehat S$ by closures of its simplices and the geometric
realization of the nerve ${\mathcal N}_V$ of the open covering of $V$ by sets $V_B$.

Since intersections of any number of elements of the covering of $\widehat S$ (i.e., simplices)
are contractible, the Nerve Theorem (Theorem~\ref{th:nerve} (ii)) implies that
$\widehat S \simeq |{\mathcal N}_{\widehat S}|$, i.e., there is a continuous map
$\psi_{\widehat S}:\> \widehat S \to |{\mathcal N}_{\widehat S}|$ which induces isomorphisms
of homotopy groups
$\psi_{\widehat S \#}:\>  \pi_k(\widehat S) \to \pi_k(|{\mathcal N}_{\widehat S}|)$ for all
integers $k \ge 0$.

On the other hand, by the Nerve Theorem (Theorem~\ref{th:nerve} (i)), there is a continuous map
$\psi_V:\> V \to |{\mathcal N}_V|$ inducing isomorphisms of homotopy groups
$\psi_{V \#}:\> \pi_k(V) \to \pi_k(|{\mathcal N}_V|)$ for every $k \le m-1$ and an
epimorphism $\psi_{V \#}:\> \pi_m(V) \to \pi_m(|{\mathcal N}_V|)$.
As $\tau_k$ take
$$
\psi^{-1}_{\widehat S \#} \circ \xi \circ \psi_{V \#}:\> \pi_k(V) \to \pi_k(\widehat S).
$$

By Whitehead Theorem (Theorem~\ref{th:Whitehead2}), $\psi_{\widehat S}$ induces isomorphisms
of homology groups
$\psi_{\widehat S \ast}:\>  H_k(\widehat S) \to H_k(|{\mathcal N}_{\widehat S}|)$ for all
$k \ge 0$, while $\psi_V$ induces isomorphisms of homology
groups $\psi_{V \ast}:\> H_k(V) \to H_k(|{\mathcal N}_V|)$ for every $k \le m-1$, and
an epimorphism $\psi_{V \ast}:\> H_m(V) \to H_m(|{\mathcal N}_V|)$.
As $\chi_k$ take
$$
\psi^{-1}_{\widehat S \ast} \circ \xi \circ \psi_{V \ast}:\> H_k(V) \to H_k(\widehat S).
$$

If $m \ge \dim (S)$ then, by Corollary~\ref{cor:combinatorics},
a nonempty intersection of any number of sets $V_B$ is contractible.
Then, according to Nerve Theorem, sets $V$ and $\widehat S$ are homotopy equivalent to
geometric realizations of the respective nerves, and therefore $V \simeq S$.
\end{proof}

\section{Proof of Theorem~\ref{th:main}}

We now need to re-define the simplicial complex $R$ so that it would satisfy additional
properties.
Recall that definable functions are triangulable \cite{Coste}, Th.~4.5.
Consider a finite simplicial complex $R'$ such that $R'$ is a triangulation of the projection
$$\rho:\> G \times [0,1] \to [0,1],$$
and $R'$ is compatible with
$$\bigcup_{\delta \in (0,1)} (S_\delta, \delta) \subset G \times [0,1].$$
Define $R$ as the triangulation induced by $R'$ on the fibre $\rho^{-1}(0)$.

\begin{definition}
Along with the sequence $\eps_0, \ldots , \delta_m$,
consider another sequence
$\eps'_0, \delta'_0, \eps'_1, \delta'_1 \ldots \eps'_m, \delta'_m $.
Let $T'$ be the set defined as in Definition~\ref{def:telescope} replacing all
$\delta_i,\ \eps_i$ by $\delta'_i,\ \eps'_i$.
Let $V'$ be the set defined as in Definition~\ref{def:V} replacing all
$\delta_i,\ \eps_i$ by $\delta'_i,\ \eps'_i$.
\end{definition}

\subsection{Definable case}

In the definable case we specify the hard--soft relation for the set $V'$ as follows.
For any pair $(\Delta_1, \Delta_2)$ of $S$
such that $\Delta_1$ is a subsimplex of $\Delta_2$, we assume that $\Delta_1$ is {\em soft}
in $\Delta_2$.

\begin{definition}\label{def:V''}
Let $B$ and $K$ be as in Lemma~\ref{le:core}.
For $0 < \eps < 1$ define
$$K_B(\eps):=
\Bigg\{ \sum_{j_\nu \in J} t_{j_\nu}v_{j_\nu}
\in K(j_0, \ldots ,j_\ell) \Bigg| \sum_{i_\nu \in I}t_{i_\nu}>1 - \eps,$$
$$\forall i_\nu \in I\> \forall j_\mu \in (J \setminus I)\ (t_{i_\nu} > t_{j_\mu}) \Bigg\}.$$
\end{definition}

Introduce a new parameter $\eps''$, and define $V''$ as the union of $K_B(\eps'')$ over
all simplices $B$ of $\widehat S$, and all simplices $K$ of $\widehat R$ such that
$B \subset \overline K$.

In each of the following Lemmas~\ref{le:V''}, \ref{le:VsubT}, \ref{le:V'eqV''}, and \ref{le:TV1},
the statement holds for
$$0 < \eps'_0 \ll \cdots \ll \eps'_i \ll \eps_i \ll \delta_i \ll \delta'_i \ll \cdots
\ll \delta'_m \ll \eps''\quad (i=0, \ldots ,m).$$

\begin{lemma}\label{le:V''}
$S \simeq V''$.
\end{lemma}

\begin{proof}
Let $B$ be a simplex of $\widehat S$, and let $U_B := \bigcup K_B(\eps'')$, where the union
is taken over all simplices $K$ of $\widehat R$ such that $B \subset \overline K$.
Then the family of all sets $U_B$ forms an open covering of $V''$ whose nerve we denote
by ${\mathcal N}_{V''}$.
Each $U_B$ is contractible, since $B$ is a deformation retract of $U_B$.
Any intersection $U:=U_{B_0} \cap \cdots \cap U_{B_k}$ is nonempty iff, after the suitable
reordering, the sequence $B_{i_1}, \ldots , B_{i_k}$ is a $k$-flag of simplices.
If $U \neq \emptyset$, then $B_{i_k}$ is its deformation retract, hence $U$ is contractible.
By Nerve Theorem (Theorem~\ref{th:nerve} (ii)), $V'' \simeq |{\mathcal N}_{V''}|$.
On the other hand, the simplices of $\widehat S$ form a covering of $S$
with nerve ${\mathcal N}_S$ (in the sense of Remark~\ref{re:simp_nerve}),
therefore, $S \simeq |{\mathcal N}_S|$ by Nerve Theorem.
Then $S \simeq V''$, since nerves ${\mathcal N}_{V''}$ and ${\mathcal N}_S$ are isomorphic.
\end{proof}

\begin{lemma}\label{le:VsubT}
$V' \subset T \subset V''$.
\end{lemma}

\begin{proof}
Let $V_{\delta,\eps}$ be the union of sets $K_B(\delta,\eps)$ over all
simplices $K$ of $\widehat R$ and simplices $B$ of $\widehat S$
such that $B\subset\overline K$.
We first show that $V_{\delta',\eps'}\subset S_{\delta,\eps}$ which immediately implies
$V' \subset T$.
Fix $\delta'$, and let $\x_{\eps'}\in V_{\delta',\eps'}$ be a definable curve.
Then $\x_0:=\lim_{\eps'\searrow 0} \x_{\eps'}\in \overline {B(\delta')}$,
where $B$ is a simplex in $\widehat S$ (this follows from Definition 2.6).
Let $\Delta$ be the simplex in $S$ containing $\x_0$.
Since every subsimplex of $\Delta$ is soft in $\Delta$, $\x_0 \in S_\delta$ for
$\delta \ll \delta'$.
Also an open neighborhood of $\x_0$ in $G$, of the size
independent of $\eps'$, is contained in $S_{\delta,\eps}$ for
$\eps\ll\delta\ll\delta'$.
Hence $\x_{\eps'}\in S_{\delta,\eps}$ for $\eps'\ll\eps\ll\delta\ll\delta'$.

Next, we show that $S_{\delta,\eps}\subset V''$, and therefore $T \subset V''$.
Fix $\delta$, and let $\x_{\eps}\in S_{\delta,\eps}$ be a definable curve.
Then $\x_0:=\lim_{\eps \searrow 0} \x_{\eps}\in S_{\delta}$.
Hence  $\x_0$ belongs to a simplex $B$ of $\widehat S$.
According to Definition~\ref{def:V''}, an open neighbourhood of $\x_0$ of the radius
larger than $\eps$ is contained in $K_B(\eps'')$ for any
simplex $K$ of $\widehat R$ such that $B \subset \overline K$.
In particular, $\x_\eps \in K_B(\eps'')$ and therefore $\x_\eps \in V''$.
\end{proof}

\begin{lemma}\label{le:V'eqV''}
The inclusion map $\iota:\> V' \hookrightarrow V''$ induces isomorphisms of homotopy groups
$\iota_{k \#}:\> \pi_k(V') \to \pi_k(V'')$ for every $k \le m-1$, and an epimorphism
$\iota_{m \#}$.
\end{lemma}

\begin{proof}
Recall that $V'$ admits an open covering  by sets of the
kind $V_B':=V_B$ (Definition~\ref{def:V}), over all simplices $B$ in $\widehat S$, such that
every non-empty intersection of sets $V_B'$ is $(m-1)$-connected (Lemma~\ref{le:V_B}).
Similarly, the set $V''$ has an open covering by sets $V_B''$,
where $V_B''$ is the union of sets $K_{B'}( \eps'')$ over all simplices
$B' \in S_B$, and simplices $K$ of $\widehat R$ such that $B \subset \overline K$.
Every non-empty intersection of sets $V_B''$ is contractible
(cf. the proof of Lemma~\ref{le:V''}).
Let ${\mathcal N}'$ and ${\mathcal N}''$ be the respective nerves of these two coverings.

The inclusion relation $V_B' \subset V_B''$ induces the canonical isomorphism
$\xi:\> |{\mathcal N}'| \to |{\mathcal N}''|$ of geometric realizations of nerves
${\mathcal N}'$ and ${\mathcal N}''$.
According to Nerve Theorem (Theorem~\ref{th:nerve} (i), details in \cite{Bjorner}) there are a
continuous maps $f':\> V' \to |{\mathcal N}'|$ and $f'':\> V'' \to |{\mathcal N}''|$,
such that for any $B$ in $\widehat S$, $f'(V'_B)$ and $f''(V''_B)$ are vertices
in ${\mathcal N}'$ and ${\mathcal N}''$ respectively, with $f''(V''_B)= \xi (f'(V'_B))$.
Also, $f'(V'_{B_1} \cap \cdots \cap V'_{B_k})$ and
$f''(V''_{B_1} \cap \cdots \cap V''_{B_k})$ are simplices spanned by sets of vertices
$\{ f'(V'_{B_1}), \ldots , f'(V'_{B_k}) \}$ and $\{ f''(V''_{B_1}), \ldots , f''(V''_{B_k}) \}$
respectively, and such that
$f''(V''_{B_1} \cap \cdots \cap V''_{B_k})= \xi (f'(V'_{B_1} \cap \cdots \cap V'_{B_k}))$.
Herewith, the map $f'$ induces isomorphisms of homotopy groups $f'_{\#}:\> \pi_k(V') \to
\pi_k(|{\mathcal N}'|)$ for every $k \le m-1$, and an epimorphism $f'_{\#}:\> \pi_m(V') \to
\pi_m(|{\mathcal N}'|)$, while $f''$ is a homotopy equivalence.
It follows that the diagram
$$
\begin{CD}
V' @> \iota > \hookrightarrow >  V''\\
@V f' V  V @V f'' V \simeq V\\
|{\mathcal N}'| @> \xi >> |{\mathcal N}''|
\end{CD}
$$
is commutative.
Since $\xi \circ f'$ induces isomorphisms of homotopy groups $\pi_k$ for
every $k \le m-1$, and epimorphisms for $k=m$, while $f''$ induces isomorphisms for any
$k \ge 0$, the inclusion map $\iota$ also induces isomorphisms for
every $k \le m-1$, and an epimorphism for $k=m$.
\end{proof}

\begin{lemma}\label{le:TV1}
For every $k \le m$, there are epimorphisms $\zeta_k:\> \pi_k(T) \to \pi_k(V'')$ and
$\eta_k:\> H_k(T) \to H_k(V'')$.
\end{lemma}

\begin{proof}
Due to Lemmas~\ref{le:VsubT}, \ref{le:V'eqV''},
$$V' \stackrel{p}{\hookrightarrow} T \stackrel{q}{\hookrightarrow} V'',$$
where $\hookrightarrow$ are the inclusion maps, and $q \circ p$ induces isomorphisms
$(q \circ p)_\#=q_\# \circ p_\#$ of
homotopy groups $\pi_k(V') \cong \pi_k(V'')$ for every $k \le m-1$, and an epimorphism
$\pi_m(V') \to \pi_m(V'')$.
Then $\zeta_k := q_\#$ is an epimorphism for every $k \le m$.

By Whitehead Theorem (Theorem~\ref{th:Whitehead2}), $q \circ p$ also induces isomorphisms
$(q \circ p)_\ast=q_\ast \circ p_\ast$ of homology groups $H_k(V') \cong H_k(V'')$ for every
$k \le m-1$, and an epimorphism $H_k(V') \to H_k(V'')$.
Hence, $\eta_k := q_\ast$ is an epimorphism for every $k \le m$.
\end{proof}

Theorem~\ref{th:main}(i) immediately follows from Lemmas~\ref{le:V''} and \ref{le:TV1}.

\subsection{Separability and constructible case}

\begin{definition}\label{def:separ}
For the simplicial complex $R$ and the family $\{ S_\delta \}$ we call the pair
$(R, \{ S_\delta \})$ {\em separable} if for any pair $(\Delta_1, \Delta_2)$ of simplices of
$S$ such that $\Delta_1$ is a
subsimplex of $\Delta_2$, the equality $\overline {\Delta_2 \cap S_\delta} \cap \Delta_1
=\emptyset$ is equivalent to the inclusion
$\Delta_1 \subset \overline {\Delta_2 \setminus S_\delta}$ for all sufficiently
small $\delta >0$.
\end{definition}

Recall that in the constructible case we assume that $S$ is defined by a Boolean combination
of equations and inequalities with continuous definable functions, and the set $S_\delta$
is defined using sign sets of these functions (see Section~\ref{sec:main}).

\begin{lemma}
In the constructible case $(R, \{ S_\delta \})$ is separable.
\end{lemma}

\begin{proof}
Observe that $R$ is compatible with the sign set decomposition of $S$.

Consider a pair $(\Delta_1, \Delta_2)$ of simplices of $S$ such that $\Delta_1$ is a subsimplex
of $\Delta_2$.
If both $\Delta_1$ and $\Delta_2$ lie in the same sign set, then
$\overline {\Delta_2 \cap S_\delta} \cap \Delta_1 = \overline{\Delta_1 \cap S_\delta}
\neq \emptyset$ and
$\Delta_1 \not\subset \overline {\Delta_2 \setminus S_\delta}$.

If $\Delta_1$ and $\Delta_2$ lie in two different sign sets, then there is a function $h$
in the Boolean combination defining $S$ such
that $h(\x)=0$ for every point $\x \in \Delta_1$, while $h(\y)$ satisfies a strict inequality,
say $h(\y)>0$, for every point $\y \in \Delta_2$.
Then $\overline {\Delta_2 \cap S_\delta} \subset \overline{ \Delta_2 \cap \{ h \ge \delta \}}$
and $\overline{ \Delta_2 \cap \{ h \ge \delta \}} \cap \{ h=0 \}= \emptyset$.
Hence $\overline {\Delta_2 \cap S_\delta} \cap \Delta_1= \emptyset$.
On the other hand, $\overline {\Delta_2 \setminus S_\delta} \supset
\overline {\Delta_2 \cap \{ h< \delta \}} \supset \overline {\Delta_2 \cap \{ h=0 \}}
\supset \Delta_1$.
\end{proof}

Now we return to the general definable case, and assume for the rest of this section
that $(R, \{ S_\delta \})$ is separable.
For any pair $(\Delta_1, \Delta_2)$ of simplices of $S$ such that $\Delta_1$ is a
subsimplex of $\Delta_2$, we assume that $\Delta_1$ is {\em soft} in $\Delta_2$ if
$\overline {\Delta_2 \cap S_\delta} \cap \Delta_1 = \emptyset$ (equivalently,
$\Delta_1 \subset \overline {\Delta_2 \setminus S_\delta}$)
for all sufficiently small $\delta >0$.
Otherwise, $\Delta_1$ is {\em hard} in $\Delta_2$.

\begin{lemma}\label{le:U_x}
If $\Delta_1$ is hard in $\Delta_2$, then for every $\x \in \Delta_1$ there is a neighbourhood
$U_\x$ of $\x$ in $\overline \Delta_2$ such that for all sufficiently small $\delta \in (0,1)$,
$U_\x \subset \overline {\Delta_2 \cap S_\delta}$.
\end{lemma}

\begin{proof}
Suppose that contrary to the claim, for some $\x \in \Delta_1$,
$U_\x \setminus \overline {\Delta_2 \cap S_\delta} \neq \emptyset$ for any neighbourhood
$U_\x$ of $\x$ in $\Delta_1$, for arbitrarily small $\delta >0$.

Since the set $S_\delta$ grows (with respect to inclusion) as $\delta \searrow 0$,
and $\Delta_1$ is hard in $\Delta_2$, the intersection
$\overline {\Delta_2 \cap S_\delta} \cap \Delta_1$ is non-empty and also grows.
If for any neighbourhood $W_\x$ of $\x$ in $\Delta_1$,
$W_\x \not\subset \overline {\Delta_2 \cap S_\delta} \cap \Delta_1$ for arbitrarily small
$\delta >0$, then the limits of both $\overline {\Delta_2 \cap S_\delta} \cap \Delta_1$
and its complement in $\Delta_1$, as $\delta \searrow 0$, have non-empty intersections with
$\Delta_1$.
This contradicts to the assumption that $\Delta_1$ is a simplex in the complex $R$ compatible
with $R'$, thus there is a neighbourhood $W_\x$ in $\Delta_1$ such that
$W_\x \subset \overline {\Delta_2 \cap S_\delta} \cap \Delta_1$ for sufficiently small
$\delta >0$.
It follows that $U_\x \setminus \overline {\Delta_2 \cap S_\delta} \subset \Delta_2$.
Since $\x \in \overline{\Delta_2 \setminus S_\delta}$, and using again the compatibility of
the complex $R$ with $R'$, we conclude that
$\Delta_1 \subset \overline{\Delta_2 \setminus S_\delta}$,
i.e., $\Delta_1$ is soft in $\Delta_2$, which is a contradiction.
\end{proof}

In each of the following Lemmas~\ref{le:TsubV}, \ref{le:ThomT'}, and Theorem~\ref{th:TV}
the statement holds for
\begin{equation}\label{eq:eps_eps'}
0 < \eps'_0 \ll \cdots \ll \eps'_i \ll \eps_i \ll \delta_i \ll \delta'_i \ll \cdots
\ll \delta'_m \ll 1 \quad (i=0, \ldots ,m).
\end{equation}

\begin{lemma}\label{le:TsubV}
$T' \subset V$ and $V' \subset T$.
\end{lemma}

\begin{proof}
We show first that $S_{\delta',\eps'}\subset V_{\delta,\eps}$,
for $\eps'\ll\eps\ll\delta\ll\delta'$, where
$V_{\delta,\eps}$ is the union of $K_B(\delta,\eps)$ over all
simplices $K$ of $\widehat R$ and simplices $B$ of $\widehat S$
such that $B\subset\overline K$.

Let us fix $\delta'$, and let $\x_{\eps'}\in S_{\delta',\eps'}$ be any definable curve.
It is enough to show that $\x_{\eps'}\in V_{\delta,\eps}$ for
$\eps' \ll \eps\ll\delta\ll\delta'$.
Clearly, $\x_0=:\lim_{\eps'\searrow 0} \x_{\eps'}$ belongs to $S_{\delta'}$.
Hence  $\x_0$ belongs to a simplex $B=B(j_0, \ldots ,j_{\ell})$ of $\widehat S$.
Suppose that $\x_0 \not\in B(\delta)$.
Let $\x_{0, \delta} \in B \setminus B(\delta)$ be a definable curve.
Then $\x_{0,0}=:\lim_{\delta \searrow 0} \x_{0, \delta}$ belongs to
a subsimplex $B'=B(i_0, \ldots ,i_k)$ of $B$.
It follows that $\x_{0,0} \in \overline{ \Delta_{j_0} \cap S_{\delta'}} \cap \Delta_{i_0}$,
therefore $\Delta_{i_0}$ is {\em hard} in $\Delta_{j_0}$.
On the other hand, by the definition of $B(\delta)$ (Definition~\ref{def:K_B}),
$\Delta_{i_0}$ is {\em soft} in $\Delta_{j_0}$.
This contradiction shows that $\x_0 \in B(\delta)$.

For $\eps'\ll\eps$, the distance from $\x_{\eps'}$ to
$\x_0\in S_{\delta'}\cap B$ is much smaller than $\eps$.
From Definition 2.6, for $\eps\ll\delta\ll\delta'$,
the union of $K(\delta,\eps)$ over all simplices $K$ of $\widehat R$
such that $B \subset\overline K$ contains
an open in $G$ neighborhood of $\x_0 \in B$ of the size that is independent of $\eps'$.
Hence $\x_{\eps'}\in V_{\delta,\eps}$ for $\eps'\ll\eps\ll\delta\ll\delta'$.

Next, we want to show that $V_{\delta',\eps'}\subset S_{\delta,\eps}$.
As before, fix $\delta'$.
Let $\x_{\eps'}\in V_{\delta',\eps'}$ be a definable curve.
Then $\x_0:=\lim_{\eps'\searrow 0} \x_{\eps'}\in B(\delta')$,
where $B=B(j_0, \ldots, j_{\ell})$ is a simplex in $\widehat S$
(this follows from Definition~\ref{def:K_B}).
Suppose that $\x_0 \not\in S_\delta$, then $\x_0 \in B \setminus S_\delta$.
Let $\x_{0, \delta} \in B \setminus S_\delta$ be a definable curve.
Therefore $\x_{0,0}:=\lim_{\delta \searrow 0} \x_{0, \delta}$ belongs to
a subsimplex $B'=B(i_0, \ldots ,i_k)$ of $B$.
Then by Lemma~\ref{le:U_x}, $\Delta_{i_0}$ is {\em soft} in $\Delta_{j_0}$, and thus
$\x_{0,0} \not\in V_{\delta',\eps'}$.
The same is true for $\x_{\eps'}$ as well, namely
$\x_{\eps'} \not\in V_{\delta',\eps'}$ for $\eps' \ll \delta \ll \delta'$.
This contradiction shows that $\x_0 \in S_\delta$.

Since $\x_0\in S$, an open neighborhood of $\x_0$ in $G$, of the size
independent of $\eps'$, is contained in $S_{\delta,\eps}$ for
$\eps\ll\delta\ll\delta'$.
Hence $\x_{\eps'}\in S_{\delta,\eps}$ for $\eps'\ll\eps\ll\delta\ll\delta'$.
\end{proof}

\begin{lemma}\label{le:ThomT'}
The inclusion maps $T' \hookrightarrow T$ and $V' \hookrightarrow V$ are homotopy equivalences.
\end{lemma}

\begin{proof}
Proofs of homotopy equivalences are similar for the both inclusions, so we will consider
only the case of $T' \hookrightarrow T$.

Consider $\eps_0, \delta_0, \ldots , \eps_m, \delta_m$ as variables, then
$T \subset \Real^{n+2m+2}$.
From the o-minimal version of Hardt's triviality, applied to the projection
$\rho:\> T \to \Real^{2m+2}$ on the subspace of coordinates
$\eps_0, \delta_0, \ldots , \eps_m, \delta_m$, follows the existence of a partition of
$\Real^{2m+2}$ into a finite number of connected definable sets $\{ A_i \}$ such that $T$ is
definably trivial over each $A_i$, i.e., for any point
$(\bar \eps, \bar \delta):= (\eps_0, \delta_0, \ldots , \eps_m, \delta_m) \in A_i$ the pre-image
$\rho^{-1}(A_i)$ is definably homeomorphic to $\rho^{-1}(\bar \eps, \bar \delta) \times A_i$ by
a fibre preserving homeomorphism.

There exists an element $A_{i_0}$ of the partition which is an open connected set in
$\Real^{2m+2}$ and
contains both points $(\bar \eps, \bar \delta)$ and $(\bar \eps', \bar \delta')$ for
(\ref{eq:eps_eps'}).
Let $\gamma:\> [0,1] \to A_{i_0}$ be a definable simple curve such that
$\gamma (0)=(\bar \eps, \bar \delta)$ and $\gamma (1)= (\bar \eps', \bar \delta')$.
Then $\rho^{-1}(\gamma (0))=T$, $\rho^{-1}(\gamma (1))=T'$ and
$\rho^{-1}(\gamma([0,1]))$ is definably homeomorphic to $T \times \gamma([0,1])$.
Let $\Phi_{t,t'}:\> \rho^{-1}(\gamma (t')) \to \rho^{-1}(\gamma (t))$ for $0 \le t \le t' \le 1$
be the homeomorphism of fibres.
Replacing if necessary $(\bar \eps, \bar \delta)$ by a point closer to
$(\bar \eps', \bar \delta')$ along the curve $\gamma$,
we can assume that $\rho^{-1}(\gamma (t')) \subset \rho^{-1}(\gamma (t))$ for all
$0 \le t \le t' \le 1$.
Then $T'$ is a strong deformation retract of $T$ defined by the homotopy
$F:\> T \times [0,1] \to T$ as follows.
If $\x \in \rho^{-1}(\gamma(t'))$ for some $t' \le t$ and $\x \not\in \rho^{-1}(\gamma(t''))$
for any $t'' > t'$, then $F(\x, t)= \Phi_{t',\ t} (\x)$.
If $\x \in \rho^{-1}(\gamma(t'))$ with $t'>t$, then $F(\x,t)=\x$.
\end{proof}

\begin{theorem}\label{th:TV}
$T \simeq V$
\end{theorem}

\begin{proof}
Consider four sequences
$(\eps^{(j)}, \delta^{(j)}):=(\eps^{(j)}_{0}, \delta^{j}_{0}, \ldots ,
\eps^{(j)}_{m}, \delta^{(j)}_{m})$.
Let $T(\eps^{(j)}, \delta^{(j)})$ (respectively, $V(\eps^{(j)}, \delta^{(j)})$)
be the set defined as in Definition~\ref{def:telescope} (resp., Definition~\ref{def:V})
replacing all $\delta_i,\ \eps_i$ by $\delta^{(j)}_{i},\ \eps^{(j)}_{i}$.

Due to Lemmas~\ref{le:TsubV}, the following chain of inclusions holds
$$T(\eps^{(1)}, \delta^{(1)}) \stackrel{p}{\hookrightarrow} V(\eps^{(2)}, \delta^{(2)})
\stackrel{q}{\hookrightarrow} T(\eps^{(3)}, \delta^{(3)}) \stackrel{r}{\hookrightarrow}
V(\eps^{(4)}, \delta^{(4)}),$$
for
$$0< \eps^{(j)}_{0} \ll \delta^{(j)}_{0}\ll \cdots \ll\eps^{(j)}_{m} \ll
\delta^{(j)}_{m} \ll 1,$$
where
$$\delta^{(j-1)}_{i-1} \ll \eps^{(j-1)}_i \ll \eps^{(j)}_i \ll \delta^{(j)}_i \ll
\delta^{(j-1)}_i$$
for all $i=1, \ldots ,m$, $j=2,3,4$.

According to Lemma~\ref{le:ThomT'}, $q \circ p$ and $r \circ q$ are homotopy equivalences.
Passing to induced homomorphisms of homotopy groups, we have that $(q \circ p)_\ast=
q_\ast \circ p_\ast$ is an isomorphism, hence $q_\ast$ is epimorphism.
Similarly, since $(r \circ q)_\ast=r_\ast \circ q_\ast$ is an isomorphism, $q_\ast$ is a
monomorphism.
It follows that $q_\ast$ is an isomorphism, therefore $T \simeq V$ by Whitehead Theorem
on weak homotopy equivalence (Theorem~\ref{th:Whitehead1}).
\end{proof}

Theorem~\ref{th:main}(ii) immediately follows from Theorems~\ref{th:TV} and \ref{th:VS}.

\section{Upper bounds on Betti numbers}

The method described in this section can be applied to obtain upper bounds on Betti
numbers for sets defined by Boolean formulae with functions from various classes
which admit a natural measure of ``description complexity'' and a suitable version of
``Bezout Theorem'', most notably for semialgebraic and semi- and sub-Pfaffian sets
(see, e.g., \cite{GV04}).
We give detailed proofs for the semialgebraic case.
The proofs can be extended to the Pfaffian case straightforwardly.

\begin{definition}
Let $f, g, h:\> \N^\ell \to \N $ be three functions, $n \in \N$.
The expression $f \le O(g)^n$ means: there exists $c \in \N$
such that $f \le (cg)^n$ everywhere on $\N^\ell$.
The expression $f \le g^{O(h)}$ means: there exists $c \in \N$ such that
$f \le g^{c h}$ everywhere on $\N^\ell$.
\end{definition}

\subsection{Semialgebraic sets defined by quantifier-free formulae}
Consider the constructible case with $S= \{ \x|\> {\mathcal F}(\x)\} \subset \Real^n$, where
${\mathcal F}$ is a Boolean combination
of polynomial equations and inequalities of the kind $h(\x)=0$ or $h(\x)>0$,
$h \in \Real[x_1, \ldots ,x_n]$.
Suppose that the number of different polynomials $h$ is $s$ and their degrees do not exceed $d$.
The following upper bounds on the total Betti number ${\rm b}(S)$ of the set $S$ originate from
the classic works of \cite{O, PO, Milnor, Thom}.
Their proofs can be found in \cite{BPR}.

\begin{itemize}
\item[(i)]
If ${\mathcal F}$ is a conjunction of any number equations, then
${\rm b}(S) \le d(2d-1)^{n-1}$.
\item[(ii)]
If ${\mathcal F}$ is a conjunction of $s$ {\em non-strict} inequalities, then
${\rm b}(S) \le (sd+1)^n$.
\item[(iii)]
If ${\mathcal F}$ is a conjunction of $s$ equations and {\em strict} inequalities, then
${\rm b}(S) \le O(sd)^n$.
\end{itemize}

The following statement applies to more general semialgebraic sets.

\begin{theorem}[\cite{Basu99}, Th. 1;\ \cite{BPR}, Th. 7.38]\label{th:alg_bound0}
If ${\mathcal F}$ is a monotone Boolean combination (i.e., exclusively connectives
$\land,\ \lor$ are used, no negations) of only strict or only
non-strict inequalities, then ${\rm b}(S) \le O(sd)^n$.
\end{theorem}

In \cite{GV05}, Th. 1 the authors proved the bound ${\rm b}(S) \le O(s^2d)^n$ for an
arbitrary Boolean formula ${\mathcal F}$.
Theorem~\ref{th:main} implies the following refinement of this bound.

\begin{theorem}\label{th:alg_bound1}
Let $\nu:= \min \{ k+1, n-k, s \}$.
Then the $k$-th Betti number
$${\rm b}_k(S) \le O(\nu sd)^n.$$
\end{theorem}

\begin{proof}
Assume first that $k>0$.
For $m=k$ construct $T(S)$ in the compactification of $\Real^n$,
as described in Section~\ref{sec:main}.
$T(S)$ is a compact set defined by a Boolean formula with $4(k+1)s$ polynomials in
$\Real[x_1, \ldots ,x_n]$ of the kind $h+ \delta_i$, $h- \delta_i$, $h+ \eps_i$
or $h- \eps_i$, $0 \le i \le k$, having degrees at most $d$.
According to Lemma~\ref{le:components}, there is a bijection $C$ from the set ${\bf T}$ of
all connected components of $T(S)$ to the set ${\bf S}$ of all connected components of $S$
such that $C^{-1} (S')=T(S')$ for every $S' \in {\bf S}$.
By Theorem~\ref{th:main} (i), ${\rm b}_k (S') \le {\rm b}_k (T(S'))$.
I follows that
$${\rm b}_k(S)= \sum_{S' \in {\bf S}}{\rm b}_k(S') \le \sum_{S' \in {\bf S}}{\rm b}_k(T(S'))=
{\rm b}_k(T(S)).$$
Then, applying the bound from Theorem~\ref{th:alg_bound0} to $T(S)$,
\begin{equation}\label{eq:betti1}
{\rm b}_k(S) \le {\rm b}_k(T(S)) \le O((k+1)sd)^n.
\end{equation}

On the other hand, since $T(S)$ is compact, ${\rm b}_k(T(S)) =
{\rm b}_{n-k-1}(\Real^n \setminus T(S))$ by Alexander's duality.
The semialgebraic set $\Real^n \setminus T(S)$ is defined by a monotone Boolean combination
of only strict inequalities, hence, due to Theorem~\ref{th:alg_bound0},
\begin{equation}\label{eq:betti2}
{\rm b}_k(S) \le {\rm b}_{n-k-1}(\Real^n \setminus T(S)) \le O((n-k)sd)^n.
\end{equation}
The theorem now follows from (\ref{eq:betti1}), (\ref{eq:betti2}) and the
bound ${\rm b}(S) \le O(s^2d)^n$ from \cite{GV05}.

In the case $k=0$, ${\rm b}_0 (S) \le {\rm b}_0 (T(S))$ since the map $C$
is surjective, hence by Theorem~\ref{th:alg_bound0},
$$
{\rm b}_0(S) \le {\rm b}_0(T(S)) \le O(sd)^n.
$$
\end{proof}

\subsection{Projections of semialgebraic sets}

Let $\rho:\> \Real^{n+r} \to \Real^n$ be the projection map, and
$S= \{ (\x, \y)|\> {\mathcal F}(\x, \y)\} \subset \Real^{n+r}$ be a semialgebraic set,
where ${\mathcal F}$ is a Boolean combination
of polynomial equations and inequalities of the kind $h(\x, \y)=0$ or $h(\x, \y)>0$,
$h \in \Real[x_1, \ldots ,x_{n}, y_1, \ldots ,y_r]$.
Suppose that the number of different polynomials $h$ is $s$ and their degrees do not exceed $d$.

Effective quantifier elimination algorithm (\cite{BPR}, Ch. 14) produces a Boolean
combination ${\mathcal F}_\rho$ of polynomial equations and inequalities, with polynomials
in $\Real[x_1, \ldots ,x_n]$, defining the projection $\rho (S)$.
The number of different polynomials in ${\mathcal F}_\rho$ is $(sd)^{O(nr)}$, and their
degrees are bounded by $d^{O(r)}$.
Then Theorem~\ref{th:alg_bound1} (or Theorem~1 in \cite{GV05}) implies that
\begin{equation}\label{eq:elim}
{\rm b}_k (\rho (S)) \le (sd)^{O(n^2r)}
\end{equation}
for any $k \ge 0$.
We now improve this bound as follows.

\begin{theorem}\label{th:alg_bound2}
The $k$-th Betti number of $\rho (S)$ satisfies the inequality
$${\rm b}_k (\rho (S)) \le \sum_{0 \le p \le k} O((p+1)(k+1)sd)^{n+(p+1)r}
\le ((k+1)sd)^{O(n+kr)}.$$
\end{theorem}

\begin{proof}
For $k=0$ the bound immediately follows from Theorem~\ref{th:alg_bound1}, so assume that $k>0$.
The set $S$ is represented by families $\{ S_\delta \}_\delta$,
$\{ S_{\delta , \eps} \}_{\delta , \eps}$ in the compactification of $\Real^{n+r}$ as
described in Section~\ref{sec:main}.
According to Lemma~\ref{le:map_F}, the projection $\rho (S)$ is represented by
families $\{ \rho(S_\delta) \}_\delta$,
$\{ \rho(S_{\delta , \eps}) \}_{\delta , \eps}$ in the compactification of $\Real^n$.
Fix $m=k$, then the set $T(\rho(S))=\rho (T(S))$ is defined.
According to Corollary~\ref{cor:spectral},
$${\rm b}_{k} (\rho (T(S))) \le \sum_{p+q=k} {\rm b}_{q}(W_{p}),$$
where
$$W_{p}=\underbrace{T(S) \times_{\rho (T(S))} \cdots
\times_{ \rho (T(S))} T(S)}_{p+1\> \>{\rm times}}.$$
The fibred product $W_{p} \subset \Real^{n+(p+1)r}$ is definable by a Boolean formula with
$$4(p+1)(k+1)s$$
polynomials of degrees not exceeding $d$.
Hence, by Theorem~\ref{th:alg_bound0},
$${\rm b}_{q}(W_{p}) \le O((p+1)(k+1)sd)^{n+(p+1)r}.$$
It follows that
\begin{equation}\label{eq:proj}
{\rm b}_{k} (T (\rho(S))) \le \sum_{0 \le p \le k} O((p+1)(k+1)sd)^{n+(p+1)r}
\le ((k+1)sd)^{O(n+kr)}.
\end{equation}
Finally, by Theorem~\ref{th:main} (i), ${\rm b}_k (\rho (S)) \le {\rm b}_{k} (T (\rho(S)))$,
which, in conjunction with (\ref{eq:proj}), completes the proof.
\end{proof}

\subsection{Semi- and sub-Pfaffian sets}

Necessary definitions regarding semi-Pfaffian and sub-Pfaffian sets can be found
in \cite{GV04, GV01} (see also \cite{Khov}).

Let $S= \{ \x|\> {\mathcal F}(\x)\} \subset (0,1)^n$ be a semi-Pfaffian set, where
${\mathcal F}$
is a Boolean combination of equations and inequalities with $s$ different Pfaffian functions
(here and in the sequel $(0,1)$ can be replaced by any, bounded or unbounded, interval).
Assume that all functions are defined in $(0,1)^n$, have a common Pfaffian
chain of order $\ell$, and degree $(\alpha, \beta)$.
A straightforward generalization of Theorem~\ref{th:alg_bound0} gives the following upper bound.

\begin{theorem}[\cite{Zell}, Th. 1;\ \cite{GV04}, Th. 3.4]\label{th:pfaff_bound0}
If ${\mathcal F}$ is a monotone Boolean combination of only strict or only
non-strict inequalities such that $\overline S \subset (0,1)^n$, then
$${\rm b}(S) \le s^n 2^{\ell (\ell -1)/2}O(n \beta + \min \{ n, \ell \} \alpha)^{n+ \ell}.$$
\end{theorem}

In conjunction with Theorem~\ref{th:main} this implies the following bound for the set $S$
defined by an {\em arbitrary} Boolean formula ${\mathcal F}$.

\begin{theorem}\label{th:pfaff_bound1}
Let $\nu:= \min \{ k+1, n-k, s \}$.
Then the $k$-th Betti number
$${\rm b}_k(S) \le (\nu s)^n 2^{\ell (\ell -1)/2}O(n \beta +
\min \{ n, \ell \} \alpha)^{n+ \ell}.$$
\end{theorem}

\begin{proof}
Analogous to the proof of Theorem~\ref{th:alg_bound1}.
\end{proof}

\begin{remark}
Unlike Theorem~\ref{th:pfaff_bound0}, the condition $\overline S \subset (0,1)^n$
is not required in Theorem~\ref{th:pfaff_bound1}, since
taking the conjunction of inequalities $0 < x_i <1$, for $i=1, \ldots ,n$,
with ${\mathcal F}$, guarantees that the closed set $T(S) \subset (0,1)^n$.
\end{remark}

Now we consider the sub-Pfaffian case.
Let $\rho:\> \Real^{n+r} \to \Real^n$ be the projection map, and
$S= \{ (\x, \y)|\> {\mathcal F}(\x, \y)\} \subset (0,1)^{n+r}$ be a semi-Pfaffian set,
where ${\mathcal F}$ is a Boolean combination of Pfaffian equations and inequalities.
Suppose that all different Pfaffian functions occurring in ${\mathcal F}$
are defined in $(0,1)^{n+r}$, have a common
Pfaffian chain of order $\ell$, their number is $s$, and their degree is $(\alpha , \beta)$.
Since the Pfaffian o-minimal structure does not admit quantifier elimination (i.e., the
projection of a semi-Pfaffian set may not be semi-Pfaffian, see \cite{GV04}),
it is not possible to apply in the Pfaffian case the same method that we used to obtain
the bound (\ref{eq:elim}).
On the other hand, the method employed in the proof of Theorem~\ref{th:alg_bound2} extends
straightforwardly to projections of semi-Pfaffian sets, and produces the following first general
singly exponential upper bound for Betti numbers of sub-Pfaffian sets.

\begin{theorem}\label{th:pfaff_bound2}
The $k$-th Betti number of $\rho (S)$ satisfies the inequality
$${\rm b}_k (\rho (S)) \le (ks)^{O(n+(k+1)r)}2^{O(k \ell)^2}
((n+(k+1)r)(\alpha + \beta))^{n+(k+1)r + k \ell}.$$
\end{theorem}

\begin{proof}
Analogous to the proof of Theorem~\ref{th:alg_bound2}.
\end{proof}

\subsection*{Acknowledgements} We thank J. McClure for useful discussions.
Part of this research was carried out during our joint visit
in Spring 2007 to Institute for Mathematics and
Its Applications at University of Minnesota, under the program Applications of Algebraic
Geometry, to which we are very grateful.

\end{document}